\documentclass[a4paper] {article}
[15pt]

   \usepackage[top=2.5cm, bottom=2.5cm, left=3.5cm, right= 3.5cm]{geometry}

\makeatletter
\renewcommand*\l@section{\@dottedtocline{1}{1.5em}{2.3em}}
\makeatother

\usepackage{amsfonts}
\usepackage{amssymb}
\usepackage[T1]{fontenc}

\usepackage{tikz}
\usetikzlibrary{calc}

\usepackage{CJK}
\usepackage{amsmath}
 \usepackage{algorithmicx}
 \usepackage[ruled]{algorithm}
 \usepackage{algpseudocode}

\usepackage{amsfonts}
\usepackage{amssymb}
\usepackage{amsthm}
\usepackage{amssymb}
\usepackage{enumerate}
\usepackage[calc]{picture}
\usepackage[all,cmtip]{xy}

\usepackage[mathscr]{eucal}
\usepackage{eqlist}

\usepackage{color}
\usepackage{abstract}
\usepackage[T1]{fontenc}

\theoremstyle{plain}
\newtheorem{theorem}{Theorem}
\newtheorem{proposition}[theorem]{Proposition}
\newtheorem{lemma}[theorem]{Lemma}

\newtheorem{example}[theorem]{Example}

\theoremstyle{definition}
\newtheorem{definition}{Definition}

\usepackage{etoolbox}
\newtheoremstyle{myrem}
 {3pt}
 {3pt}
 {\normalsize}
 { }
 {\itshape}
 {:}
 { }
 {}

 \theoremstyle{myrem}
 \newtheorem{remark}{Remark}
 \appto\remark{\leftskip\parindent}
 \appto\remark{\rightskip\parindent}

\numberwithin{equation}{section}
\numberwithin{theorem}{section}

\begin{document}

\begin{center}
{\Large {\textbf {A  Discrete  Morse Theory for  Digraphs}}}
 \vspace{0.58cm}

Chong Wang*\dag, Shiquan Ren*

\bigskip

\bigskip

    \parbox{24cc}{{\small
{\textbf{Abstract}.}
Digraphs are generalizations of graphs in which each edge is assigned with a direction or two directions. In this paper, we define discrete Morse functions on digraphs, and prove that  the homology of the Morse complex and the path homology are isomorphic for a transitive digraph. We also study the collapses defined by discrete gradient vector fields. Let $G$ be a digraph and $f$ a discrete Morse function. Assume the out-degree and in-degree of any zero-point of $f$ on $G$ are both 1. We prove that the original digraph $G$ and its $\mathcal{M}$-collapse $\tilde{G}$  have the same path homology groups.

}}

 \end{center}

\vspace{1cc}

\footnotetext[1]
{ {\bf 2010 Mathematics Subject Classification.}  	Primary  55U15  55N35;  Secondary  55U35.
}

\footnotetext[2]{{\bf Keywords and Phrases.}   digraphs,  Morse function,  discrete gradient vector fields, acyclic matching,  $\mathcal{M}$-collapse. }

\footnotetext[3] { * first authors;   \dag   corresponding author. }

\section{Introduction}\label{se-1}
Digraph is an important topology model of complex networks, and its topological structure has important research value and wide application prospect in data science research.  For example, J. Bang-Jensen and G.Z. Gutin \cite{dig} studied digraphs and gave applications of digraphs in quantum mechanics,  finite automata,  deadlocks of computer processes, etc.


The path homology of digraphs was first defined and studied by A. Grigor'yan, Y. Lin, Y. Muranov and S.T. Yau  in \cite{9} in 2012. Subsequently, in 2014, it is proved in \cite{10} that the path homology is functorial with respect to morphisms of digraphs, and is invariant up to certain homotopy relations of these morphisms. In 2015,  A. Grigor'yan, Y. Lin, Y. Muranov and S.T. Yau  \cite{3,4}  studied the cohomology of digraphs and graphs by using the path homology theory. Moreover, in 2017, A. Grigor'yan, Y. Muranov, and S.T. Yau \cite{2} proved some K\"{u}nneth formulas  for the path homology with coefficients in a field. In 2018,  A. Grigor'yan, Y. Muranov, V. Vershinin and S.T. Yau \cite{yau2} generalized the path homology theory of digraphs and constructed the path homology theory of multigraphs and quivers.

Discrete Morse theory originated from the study of homology groups and cell structure of cell complexes. By reducing the number of cells and simplices, discrete Morse theory is helpful to simplify the calculation of homology groups. Moreover, discrete Morse theory  provides theoretical supports for
the calculation of persistent homology in the field of topological data analysis \cite{kannan,lewiner,nanda}. In 1998, R. Forman \cite{forman1} invented the discrete Morse theory for simplicial complexes or general cell complexes, as a discrete version of the classical Morse theory of smooth Morse functions. In the subsequent references \cite{forman3,forman2,witten}, R. Forman studied the discrete Morse theory, the cup product of cohomology,  and Witten Morse theory based on \cite{forman1}. In 2005, Dmitry N. Kozlov \cite{chain} extended the combinatorial Morse complex construction to arbitrary free chain complexes, and give a short, self-contained, and elementary proof of the quasi-isomorphism between the original chain complex and its Morse complex. From 2007 to 2009, R. Ayala et al. \cite{ayala1,ayala2,ayala3,ayala4} studied the discrete Morse theory on graphs  by using the discrete Morse theory of cell complexes and simplicial complexes given by R. Forman.


Let $R$ be an arbitrary commutative ring with a unit.  Let $G=(V,E)$ be a digraph.  An {\it $n$-path} is a sequence $v_0 v_1\cdots v_n$ of $n+1$ vertices in $V$ where for each $1 \leq i \leq n$, $v_{i-1}$ and $v_i$ are assumed to be distinct. Let $\Lambda_n(V)$ be the free $R$-module consisting of all the formal linear combinations of the $n$-paths on $V$.  Let $\partial_n=\sum_{i=0}^n (-1)^i d_i$. Then $\partial_n$ is an $R$-linear map from $\Lambda_n(V)$ to $\Lambda_{n-1}(V)$ satisfying $\partial_{n}\partial_{n+1}=0$ for each $n\geq 0$ (cf. \cite{9,10,4,2}).  Hence $\{\Lambda_n(V),\partial_n\}_{n\geq 0}$ is a chain complex. An {\it allowed elementary  $n$-path} is a $n$-path such that for each $1\leq i\leq n$, $v_{i-1}\to v_i$ is a direct edge of $G$.   For simplicity, we assume that in an allowed elementary $n$-path, $v_{i-1}\not=v_i$ for each $1\leq i \leq n$. Let $P_n(G)$ be the free $R$-module consisting of all the formal linear combinations of allowed elementary $n$-paths on $G$. Then $P_n(G)$ is a  sub-$R$-module of $\Lambda_n(V)$. Note that under the boundary operator $\partial$,  the image  of an allowed elementary path does not have to be allowed.  Hence $\partial$ may not map $P_n(G)$  to $P_{n-1}(G)$. Nevertheless,  $P_n(G)$ has the sub-$R$-module $\Omega_n(G)$ which which consists of all the $\partial$-invariant elements in $P_n(G)$. We define the path homology of $G$ as the homology  of chain complex $\{\Omega_n(G),\partial_n\}$.


In this paper, we define the discrete Morse functions on digraphs in Section~\ref{se-3},  and analysis the basic properties in Section~\ref{se-4}.
In Section~\ref{se-6}, we study the discrete gradient vector fields and Morse complexes of digraphs. Let $G$ be a transitive digraph. Then we give  the first main result  of  this paper in the following theorem.
\begin{theorem}
\label{th-0.1}
 Let $G$ be a transitive digraph. Let $R$ be an arbitrary commutative ring with a unit. Then the chain complex $\Omega_*(G)$ can be decomposed as a direct sum of the Morse complex and atom chain complexes:
 \begin{eqnarray}\label{eq-5.1}
\Omega_{\ast}(G)=C_{*}^{\mathcal{M}}(B_*(G))\bigoplus\text{(direct sum of atom chain complexes)}.
\end{eqnarray}
And for each $n\geq 0$, the path homology of $G$ is isomorphic to the homology of the Morse complex
\begin{eqnarray}\label{eq-5.2}
 H_{n}(\{\Omega_*(G),\partial_*\};R)\cong  H_{n}(\{C_{*}^{\mathcal{M}}(B_*(G)),\partial_*^{\mathcal{M}}\};R)
 \end{eqnarray}
 where the explicit formula of $\partial_*^{\mathcal{M}}$ is given in Definition~\ref{def-6.1}.
\end{theorem}
Let $f:V(G)\longrightarrow [0,+\infty)$ be a discrete Morse function on $G$.
We  substitute the ordered triple of vertices $(u,v,w)$  satisfying  $f(v)=0$, $u\to v\to w$ and $u\to w$ in $G$ with the ordered couple of vertices $(u,w)$, and remove the directed edges $u\to v$ and $v\to w$. The resulting digraph $\tilde G$ and the restriction of $f$ on $\tilde G$ (denoted by $\tilde f$) form a pair $(\tilde G, \tilde f)$, which is called the {\it $\mathcal{M}$-collapse} of the pair $(G,f)$.
In Section~\ref{se-7}, we give the second main result in the following theorem.
\begin{theorem}
\label{th-0.2}
Let $G$ be a digraph and $f: V(G)\longrightarrow [0,+\infty)$ a discrete Morse function on $G$. Assume the out-degree and in-degree of any zero-point of $f$ on $G$ are both 1. Then the natural inclusion map $i:\tilde{G}\to G$ induces an isomorphism from the path homology groups of $\tilde{G}$ to the path homology groups of $G$,  where  $\tilde{G}$ is obtained by $\mathcal{M}$-collapse of the pair $(G,f)$.
\end{theorem}

Particularly, a graph is a digraph where each edge is assigned with two directions.  
Since all discrete Morse functions on  graphs have strict positive values on all the vertices,  the (combinatorial) discrete gradient vector fields are all empty. Hence, both of the two main results are trivial for graphs.

\section{Preliminaries}\label{se-2}
In this section, we mainly review some basic concepts and theorems in \cite{dig,9,10} and give some properties of digraphs.

\subsection{Digraphs and Path homology}

A {\it digraph} $G=(V,E)$ is a couple of a set $V$ whose elements
are called the vertices, and a subset $E\subset \{V\times V \setminus\text{diag}\}$ of ordered pairs of vertices whose elements are
called directed edges or arrows. The  directed edge with starting point $v$ and  ending point $w$ is  denoted by $v \to  w$.  Triangles and squares are simple examples of digraphs. A {\it triangle} is a sequence of three distinct vertices $v_0,v_1,v_2\in V$ such that $v_0\to v_1, v_1\to v_2, v_0\to v_2$. A {\it square} is a sequence of four distinct $v_0,v_1,v_1',v_2\in V$,such that $v_0\to v_1, v_1\to v_2, v_0\to v_1', v_1'\to v_2$ (cf. \cite[Section~4.2]{9}).

We define
\begin{eqnarray*}
\Omega_n(G)&=& P_n(G) \cap \partial_n^{-1} P_{n-1}(G),\\
\Gamma_n(G)&=& P_n(G)+ \partial_{n+1} P_{n+1}(G).
\end{eqnarray*}
Then as graded $R$-modules,
\begin{eqnarray}\label{eq-2.1}
\Omega_*(G) \subseteq  P_*(G)\subseteq  \Gamma_*(G) \subseteq \Lambda_*(V).
\end{eqnarray}
And as chain complexes,
\begin{eqnarray*}
\{\Omega_n(G), \partial_n\mid_{\Omega_n(G)} \}_{n\geq 0} \subseteq  \{\Gamma_n(G), \partial_n\mid_{\Gamma_n(G)} \}_{n\geq 0}\subseteq \{\Lambda_n(V),\partial_n\}_{n\geq 0}.
\end{eqnarray*}
By \cite[Proposition~2.4]{hypergraph}, the canonical inclusion
\begin{eqnarray*}
\iota:  \Omega_n(G)  \longrightarrow  \Gamma_n(G), ~~~ n\geq 0
\end{eqnarray*}
of chain complexes  induces an isomorphism between the homology groups
\begin{eqnarray*}
\iota_*: H_m(\{\Omega_n(G), \partial_n\mid_{\Omega_n(G)} \}_{n\geq 0})\overset{\cong}{\longrightarrow} H_m(\{\Gamma_n(G), \partial_n\mid_{\Gamma_n(G)} \}_{n\geq 0}), ~~~ m\geq 0,
\end{eqnarray*}
which is  the {\it path homology} of $G$. The definition of path homology here  is  essentially consistent with \cite[Definition~3.12]{9}.

\begin{definition}\cite[Section~2.3]{dig}\label{trans}
 A digraph $G$ is called {\it transitive}, if for any two directed edges $u\to v$ and $v\to w$ of $G$, there is a directed edge $u\to w$ of $G$.
\end{definition}

\begin{proposition}\label{prop-2.1}
Let $G$ be a transitive digraph. Then  for each $n\geq 0$, we  have that $\Omega_n(G)=P_n(G)$.
\end{proposition}
\begin{proof}
By (\ref{eq-2.1}), $\Omega_n(G)\subseteq P_n(G)$. On the other hand, for any  allowed elementary $n$-path $\alpha=v_0\cdots v_n$, by Definition~\ref{trans},  we have that
\begin{eqnarray*}
d_i(\alpha)=v_0\cdots v_{i-1}\hat{v_i}v_{i+1}\cdots v_n, 0\leq i\leq n
\end{eqnarray*}
is still allowed in $G$. That is, $\partial_n\alpha$ is a linear combination of allowed elementary $(n-1)$-paths. Hence, $\Omega_n(G)\subseteq P_n(G)$. The proposition follows.
\end{proof}

\subsection{Morphisms and homotopy of digraphs}
Let $G$ and $G'$  be  digraphs.  A {\it morphism} (or  {\it digraph map}) is a map $f: V(G)\longrightarrow V(G')$ such that for any vertices $u,v\in V(G)$, if $u\to v$ is a directed edge of $G$,  then  either $f(u)=f(v)$ or $f(u)\to f(v)$ is a directed edge of $G'$ (cf. \cite[Definition~2.2]{10}).  We denote such a morphism shortly as $f: G\longrightarrow G'$.  A morphism $f: G\longrightarrow G'$  is called an {\it isomorphism} if $f$ is  a bijection from $V(G)$ onto $V(G')$, and the inverse of $f$ is also a morphism.

A {\it line digraph} $I_n$  is a digraph whose the set of   vertices is $\{v_0,v_1,\ldots,v_n\}$ and the set of directed edges contains exactly one of the directed edges $v_i\to v_{i+1}$, $v_{i+1}\to v_i$ for each $i=0,1,\ldots, n-1$, and no other directed edges  (cf. \cite[p. 632]{10}).   Note that a path is a special line digraph with all the directed edges $v_i\to v_{i+1}$. 

Let $G=(V(G),E(G))$ and $H=(V(H),E(H))$ be two digraphs.  Define the {\it Cartesian product} $G\boxdot H$  as a digraph with the set of vertices $V(G)\times V(H)$ and the set of directed edges as follows:  for any $x,x'\in V(G)$ and any $(y,y')\in V(H)$,  we have $(x,y)\to (x',y')$ if and only if either $x=x'$ and $y\to y'$, or $x\to x'$ and $y=y'$ (cf. \cite[Definition~2.3]{10}).

\begin{definition}(cf. \cite[Definition~3.1]{10})\label{def-2.10}
Two morphisms $f,g: G\longrightarrow H$ are called {\it homotopic} and denoted as $f\simeq g$ if there exists a line digraph $I_n$ with $n\geq 1$ and a morphism  $F: G\boxdot I_n\longrightarrow H$ such that $F\mid_{G\boxdot \{0\}}= f$ and  $F\mid_{G\boxdot \{n\}}= g$.  
We call $F$ a {\it homotopy}.
\end{definition}

Two digraphs $G$ and $H$ are called {\it homotopy equivalent} if there exist morphisms $f: G\longrightarrow H$ and $g: H\longrightarrow G$ such that $f\circ g\simeq \text{id}_{H}$ and $g\circ f\simeq \text{id}_G$. We shall write $G \simeq H$ (cf. \cite[Definition~3.2]{10}).

\begin{definition}(cf. \cite[Definition~3.4]{10})\label{def-2.11}
A  {\it retraction} of $G$ onto $H$ is a digraph map $r: G \to H$ such that $r\mid_{H}=\text{id}_{H}$.  A retraction $r : G \to H$ is called a {\it deformation retraction} if $i\circ r\simeq \text{id}_G$, where $i: H \to G$ is the natural inclusion map.
\end{definition}

\begin{proposition}\cite[Proposition~3.5]{10}\label{prop-2.2}
Let $r : G \to H$ be a deformation retraction. Then $G \simeq H$ and the maps $r$ and $i$ are homotopy inverses of each other.\qed
\end{proposition}

\begin{theorem}\cite[Theorem~3.3]{10}\label{th-2.1}
Let $f\cong g: G\to H$ be two homotopic digraph maps. Then these maps induce the identical homomorphisms of path homology groups of $G$ and $H$. Consequently, if the digraphs $G$ and $H$ are homotopy equivalent, then their path homology groups are isomorphic.\qed
\end{theorem}

\smallskip

\subsection{Some properties of digraphs}
\begin{lemma}\label{le-1.3}
Suppose $\alpha^{(n+1)}> \gamma^{(n)} > \beta^{(n-1)}$ are allowed elementary paths on $G$.
The starting points  of  $\alpha^{(n+1)}$, $\gamma^{(n)}$  and $\beta^{(n-1)}$  are the same, and the end points of them are also the same. 
Then either of the followings holds:
\begin{enumerate}[(a).]
\item
there exists an allowed elementary $n$-path $\gamma'^{(n)}\neq \gamma^{(n)}$ such that $\alpha > {\gamma'}  > \beta $;
\item
$\beta$ is obtained by removing two subsequent vertices $v_i\to v_{i+1}$ in $\alpha$ where $1\leq i\leq n-1$.
\end{enumerate}
\end{lemma}

\begin{proof}
Without loss of generality,  we write $\alpha=v_0v_1\cdots v_{n+1}$ and $\beta=v_0\cdots \widehat{v_i}\cdots \widehat{v_j}\cdots v_{n+1}$. Since $\alpha$ and $\beta$ have the same start points and the same end points,  we have $1\leq i<  j\leq n$.

{\sc Case~1}. $ i<j-1$.

{\sc Subcase~1.1}.
$\gamma= v_0\cdots \widehat{v_i}\cdots  v_j \cdots v_{n+1}$.

Then we let $\gamma'= v_0\cdots v_i\cdots  \widehat{v_j} \cdots v_{n+1}$.  Since $\beta$ is an allowed elementary path and $i\neq j-1$,  we have that $v_{j-1}\to v_{j+1}$ is a directed path in $G$. Hence $\gamma'$ is an allowed elementary path.

{\sc Subcase~1.2}.
$\gamma= v_0\cdots v_i\cdots  \widehat{v_j} \cdots v_{n+1}$.

Then we let $\gamma'= v_0\cdots \widehat{v_i}\cdots  v_j \cdots v_{n+1}$.  Since $\beta$ is an allowed elementary path and $i+1\neq j$,  we have that $v_{i-1}\to v_{i+1}$ is a directed path in $G$. Hence $\gamma'$ is an allowed elementary path.

In both {\sc Subcase~1.1} and {\sc Subcase~1.2}, we can see that $\gamma'^{(n)}\neq \gamma^{(n)}$ and $\alpha > {\gamma'}  > \beta $.

{\sc Case~2}. $i=j-1$.

Then $\beta$ is obtained by removing two subsequent vertices $v_i\to v_{i+1}$ in $\alpha$ and $1\leq i\leq n-1$.

Summarising Case 1 and Case 2, the lemma follows.
\end{proof}

\begin{lemma}\label{le-1.4}
Suppose $\alpha^{(n+1)}> \gamma^{(n)} > \beta^{(n-1)}$ are allowed elementary paths on $G$.  The starting points  of  $\alpha^{(n+1)}$, $\gamma^{(n)}$  and $\beta^{(n-1)}$  are not all the same, or the end points of them are not all the same. 
Then either of the followings holds:
\begin{enumerate}[(a).]
\item
there exists an allowed elementary $n$-path $\gamma'^{(n)}\neq \gamma^{(n)}$ such that $\alpha > {\gamma'}  > \beta $;
\item
$\beta$ is obtained by removing $v_0\to v_{1}$ or $v_n\to v_{n+1}$  in $\alpha$ .
\end{enumerate}
\end{lemma}
\begin{proof}
It is sufficient to consider the following two cases.

{\sc Case~1}.  $\alpha$  and $\gamma$ have the same start points and the same end points. Let $\alpha=v_0\cdots v_{n+1}$, $\gamma=v_0\cdots \widehat{v_i}\cdots v_{n+1}$ for some $1\leq i\leq n$.

{\sc Subcase~1.1}.  $\beta=\widehat{v_0}\cdots \widehat{v_i}\cdots v_{n+1}$.

Then we let $\gamma'= \widehat{v_0}v_1\cdots  \cdots v_{n+1}$.  Obviously, $\gamma'$ is an allowed elementary path and $\alpha>\gamma'>\beta$, $\gamma'\not=\gamma$.

{\sc Subcase~1.2}.  $\beta=v_0\cdots \widehat{v_i}\cdots \widehat{v_{n+1}}$.

Then we let $\gamma'=v_0v_1\cdots \widehat{v_{n+1}}$.  Obviously, $\gamma'$ is an allowed elementary path and $\alpha>\gamma'>\beta$, $\gamma'\not=\gamma$.

{\sc Subcase~1.3}.  $\beta=v_0\cdots \widehat{v_i}\cdots \widehat{v_j}\cdots{v_{n+1}}$ for  $1\leq i\not=j\leq n$.

Then $\alpha$, $\gamma$ and $\beta$ have the same start points and the same end points. This is impossible.

Summarizing {\sc Subcase~1.1}, {\sc Subcase~1.2} and {\sc Subcase~1.3},  we can see that $\gamma'^{(n)}\neq \gamma^{(n)}$ and $\alpha > {\gamma'}  > \beta $.

{\sc Case~2}.  The starting points of  $\alpha$ and $\gamma$  are different, or the end points of them are different.   

{\sc Subcase~2.1}. $\alpha=v_0\cdots v_{n+1}$, $\gamma=\widehat{v_0}\cdots v_i\cdots v_{n+1}$ and $\beta=v_1\cdots \widehat{v_i}\cdots v_{n+1}$ for some $1\leq i \leq{n+1}$. We separate Subcase~2.1 into the following subcases:

 \begin{itemize}
 \item
 $i={n+1}$.

 Then $\beta=v_1\cdots v_n$. Let  $\gamma'=v_0v_1\cdots  \widehat{v_{n+1}}$.  We have that $\gamma'$ is an allowed elementary path such that $\alpha>\gamma'>\beta$, $\gamma'\not=\gamma$.
  \item
  $2\leq i<{n+1}$.

  Then we let $\gamma'=v_0\cdots  \widehat{v_{i}}\cdots v_{n+1}$. Since $\beta$ is an allowed elementary path and $v_{i-1}\to v_{i+1}$ is  a directed path in $G$, $\gamma'$ is an allowed elementary path such that  $\alpha>\gamma'>\beta$ and $\gamma'\not=\gamma$.
  \item
  $i=1$.

  Then $\beta$ is obtained by removing two subsequent vertices $v_0\to v_1$ in $\alpha$.
  \end{itemize}

{\sc Subcase~2.2}.  $\alpha=v_0\cdots v_{n+1}$, $\gamma=v_0\cdots v_i\cdots \widehat{v_{n+1}}$ and $\beta=v_0v_1\cdots \widehat{v_i}\cdots v_{n}$ for some $0\leq i \leq{n}$. We separate Subcase 2.2 into the following subcases:

 \begin{itemize}
 \item
 $i=0$.

 Then $\beta=v_1\cdots v_n$. Let  $\gamma'=\widehat{v_0}v_1\cdots  v_{n+1}$.  Then $\gamma'$  is an allowed elementary path such that $\alpha>\gamma'>\beta$, $\gamma'\not=\gamma$.
  \item
  $0<i<{n}$.

  Then we let $\gamma'=v_0\cdots  \widehat{v_{i}}\cdots v_{n+1}$. Since $\beta$ is an allowed elementary path and $v_{i-1}\to v_{i+1}$ is  a directed path in $G$,  $\gamma'$ is an allowed elementary path such that  $\alpha>\gamma'>\beta$ and $\gamma'\not=\gamma$.
  \item
  $i=n$.

  Then $\beta$ is obtained by removing two subsequent vertices $v_n\to v_{n+1}$ in $\alpha$.
  \end{itemize}

Summarising Case 1 and Case 2, the lemma follows.
\end{proof}


\begin{proposition}\label{pr-1.1}
Suppose $\alpha^{(n+1)}> \gamma^{(n)} > \beta^{(n-1)}$ are allowed elementary paths on $G$.  Then either of the followings holds:
\begin{enumerate}[(a).]
\item
there exists an allowed elementary $n$-path $\gamma'^{(n)}\neq \gamma^{(n)}$ such that $\alpha > {\gamma'}  > \beta $;
\item
$\beta$ is obtained by removing two subsequent vertices $v_i\to v_{i+1}$ in $\alpha$ where $0\leq i \leq n$.
\end{enumerate}
\end{proposition}
\begin{proof}
By Lemma~\ref{le-1.3} and Lemma~\ref{le-1.4}, the proposition follows directly.
\end{proof}

The following example shows that for any $\alpha^{(n+1)}> {\gamma}^{(n)} >\beta^{(n-1)}$, there may not exist an
allowed elementary path ${\gamma'}^{(n)}$ such that $\alpha^{(n+1)}> {\gamma'}^{(n)} >\beta^{(n-1)}$.
\begin{example}\label{ex-3}
 Let $V=\{v_0,v_1,v_2,v_3\}$. Let $G$ be a digraph with the set of vertices $V$ and
the set of directed edges $\{v_0\to v_1,v_0\to v_2, v_0\to v_3,v_1\to v_2, v_2\to v_3 \}$.
Let $\alpha^{(3)}=v_0\to v_1\to v_2\to v_3$, $\gamma^{(2)}=v_0\to v_2\to v_3$, and $\beta^{(1)}=v_0\to v_3$.
 
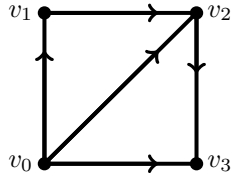
\begin{center}
\begin{tikzpicture}[line width=1.5pt]
\coordinate [label=left:$v_0$]  (A) at (2,2);
\coordinate [label=left:$v_1$]  (B) at (2,4);
 \coordinate [label=right:$v_2$]  (C) at (4,4);
  \coordinate [label=right:$v_3$]  (D) at (4,2);
   \draw [line width=1.5pt] (A) -- (B);
 \draw[->] (2,2) -- (2,3.5);
   \draw [line width=1.5pt] (A) -- (C);
 \draw[->] (2,2) -- (3.5,3.5);
   \draw [line width=1.5pt] (A) -- (D);
 \draw[->] (2,2) -- (3.5,2);
  \draw [line width=1.5pt] (B) -- (C);
 \draw[->] (2,4) -- (3.5,4);
  \draw [line width=1.5pt] (C) -- (D);
 \draw[->] (4,4) -- (4,3.2);
\fill (2,2) circle (2.5pt) (2,4) circle (2.5pt)  (4,2) circle (2.5 pt)  (4,4) circle (2.5 pt);
\end{tikzpicture}
\end{center}
We have that there is no allowed elementary $2$-path $\gamma'^{(2)}\neq \gamma^{(2)}$ such that  $\alpha^{(3)}> {\gamma'}^{(2)} >\beta^{(1)}$, where $\beta^{(1)}$ is obtained   by removing two subsequent vertices $v_1\to v_2$ in $\alpha^{(3)}$.
\end{example}

\section{Discrete Morse functions on digraphs}\label{se-3}
In this section, we define the discrete Morse functions on digraphs and the critical  allowed elementary paths of
the discrete Morse functions.

\begin{definition}\label{def-1.1}
 A  map $f: V(G)\longrightarrow [0,+\infty)$ is called a  {\it discrete Morse funtion} on $G$, if for any allowed elementary path $v_0v_1\cdots v_n$ on $G$, both of the followings hold:
\begin{quote}
\begin{enumerate}[(i).]
\item
there exists at most one $v_i$ ($0\leq i\leq n$) with $f(v_i)=0$  such that $v_0\cdots \hat{v_i} \cdots v_n$ is an allowed elementary $(n-1)$-path on $G$; 
\item
there exists at most one $u\in V(G)$ with  $f(u)=0$ such that for some $-1\leq j\leq n$,  
 $v_0\cdots v_j u v_{j+1} \cdots v_n$ (Specially, $j=-1$, $v_0\cdots v_j u v_{j+1} \cdots v_n=uv_0\cdots v_n$; $j=n$, $v_0\cdots v_j u v_{j+1} \cdots v_n=v_0\cdots v_n u$) is an allowed elementary $(n+1)$-path on $G$.   
\end{enumerate}
\end{quote}
\end{definition}

In an allowed elementary path $v_0v_1\cdots v_n$, the vertex $v_i$ is {\it removable} if  $i=0$, or $i=n$, or  $1\leq i\leq n-1$ and $v_{i-1}\to v_{i+1}$ is a directed edge in $G$.  For example, in the digraph $\{v_0\to v_1, v_1\to v_2, v_0\to v_2, v_2\to v_0\}$, $v_1$ is removable in the path $v_0v_1v_2$,   while $v_2$ is not removable in the path $v_0v_2v_0$ or   the path $v_1v_2v_0$.    Definition~\ref{def-1.1}~(i) can be restated as follows:
\begin{quote}
(i).  in an allowed elementary path, at most one removable vertex has  zero value.
\end{quote}
For the allowed elementary path $v_0v_1\cdots v_n$, a vertex $u$ of $G$ is called {\it addable}, if $u\to v_0$ is a directed edge,  or $v_n\to u$ is a directed edge, or there exists $1\leq i\leq n$ such that both $v_{i-1}\to u$ and $u\to v_{i+1}$ are directed edges.
 Definition~\ref{def-1.1}~(ii) can be restated as follows:
 \begin{quote}
(ii).   for an allowed elementary path, at most one addable vertex  of  $G$ has  zero value.
\end{quote}

Let $f$ be a non-negative  function on $V(G)$.  For any allowed elementary path
$v_0v_1\cdots v_n$,  we define the value of $f$ on the path by letting
\begin{eqnarray}\label{eq-1.1}
f(v_0v_1\cdots v_n)= \sum_{i=0}^n f(v_i).
\end{eqnarray}
We use $\gamma^{(n)}$ (or $\gamma$ for short) to denote the allowed elementary $n$-path. For $0\leq i\leq n$, if  $v_0 \cdots \hat{v_i}\cdots v_n$ is   an allowed elementary $(n-1)$-path, then we write this $(n-1)$-path as $d_i \gamma$.  Note that $\hat{v_0} \cdots  {v_i}\cdots v_n$  and  $v_0 \cdots {v_i}\cdots \hat{v_n}$ are always allowed elementary $(n-1)$-paths. Hence we always have $d_0\gamma$ and $d_n\gamma$.
By (\ref{eq-1.1}),  Definition~\ref{def-1.1}~(i) can be  restated  as  follows:
\begin{quote}
(i).  there exists at most one allowed elementary $(n-1)$-path $\beta^{(n-1)}$ such that $f(\beta)=f(\gamma)$ and $d_i\gamma=\beta$ for some $0\leq i\leq n$;
\end{quote}
and Definition~\ref{def-1.1}~(ii) can be  restated  as  follows:
\begin{quote}
(ii).  there exists at most one allowed elementary $(n+1)$-path $\alpha^{(n+1)}$ such that $f(\alpha)=f(\gamma)$ and $d_i\alpha=\gamma$ for some $0\leq i\leq n+1$.
\end{quote}

For any allowed elementary paths $\gamma$ and $\gamma'$, if $\gamma'$ can be obtained from $\gamma$ by removing some vertices, then we write $\gamma'<\gamma$ or $\gamma>\gamma'$.
For any allowed elementary $n$-path $\gamma^{(n)}$ on $G$,  we can rewrite (i), (ii) in Definition~\ref{def-1.1}  equivalently as the following inequalities
\begin{quote}
\begin{enumerate}[(i).]
\item
$\#\Big\{\beta^{(n-1)}<\gamma^{(n)}\mid f(\beta)=f(\gamma)\Big\}\leq 1$;
\item
$\#\Big\{\alpha^{(n+1)}>\gamma^{(n)}\mid f(\alpha)=f(\gamma)\Big\}\leq 1$.
\end{enumerate}
\end{quote}
For an allowed elementary path $\gamma$, if in both (i) and (ii), the inequalities hold strictly, then $\gamma$ is called critical.  Precisely,
\begin{definition}\label{def-1.2}
An allowed elementary $n$-path $\gamma^{(n)}$ is called {\it critical}, if    both of the followings hold:
\begin{quote}
\begin{enumerate}[(i)']
\item
$\#\Big\{\beta^{(n-1)}<\gamma^{(n)}\mid f(\beta)=f(\gamma)\Big\}=0$,
\item
$\#\Big\{\alpha^{(n+1)}>\gamma^{(n)}\mid f(\alpha)=f(\gamma)\Big\}=0$.
\end{enumerate}
\end{quote}
\end{definition}

By Definition~\ref{def-1.2},  an allowed elementary $n$-path $\gamma^{(n)}$ on $G$ is {\bf non-critical}  if and only if either of the followings holds:
\begin{quote}
\begin{enumerate}[(i)'']
\item
there exists $\beta^{(n-1)}<\gamma^{(n)}$ such that $f(\beta)=f(\gamma)$;
\item
there exists $\alpha^{(n+1)}>\gamma^{(n)}$ such that $f(\alpha)=f(\gamma)$.
\end{enumerate}
\end{quote}

\section{Properties of Discrete Morse functions on Digraphs}\label{se-4}

 In this section, we  prove some auxiliary results and some additional properties about discrete Morse functions and the critical allowed elementary paths of $G$.

\subsection{Auxiliary results for main theorems}
A {\it directed loop} in $G$  is an allowed elementary path $v_0 v_1\ldots v_nv_0$ for $n\geq 1$ $^{[1]}$.

\footnotetext[1] {[1].  In \cite[Definition~4.3]{10}, a {\it loop} on a digraph $G$ is defined to be a based map from a line digraph to $G$ such that the start vertex and the end vertex and  the same.    Our {directed loop} here is different from the loop in \cite{10}. A directed loop is a special loop and the converse is not true. }

\begin{lemma}\label{le-1.11}
For any digraph $G$ and any discrete Morse function $f$ on $G$,  if $v_0 v_1\ldots v_nv_0$ ($n\geq 1$)  is a directed loop in $G$,  then $f(v_i)>0$ strictly for any $0\leq i\leq n$.
\end{lemma}
\begin{proof}
Let $\alpha=v_0 v_1\ldots v_nv_0$  be  a directed loop in $G$.  Suppose to the contrary, $f(v_i)=0$ for some $0\leq i\leq n$.
Suppose $\beta=v_iv_{i+1}\cdots v_n v_0\cdots v_{i-1}v_i$, $\gamma=v_iv_{i+1}\cdots v_n v_0\cdots v_{i-1}$ and $\gamma'=v_{i+1}\cdots v_n v_0\cdots v_{i-1}v_i$. Then $\gamma <\beta, \gamma'<\beta$ and $f(\beta)=f(\gamma)=f(\gamma')$. This contradicts that $f$ is a discrete Morse function on $G$.  Therefore,   $f(v_i)>0$ strictly for any $0\leq i\leq n$.
\end{proof}

The next lemma follows from Lemma~\ref{le-1.11}.

\begin{lemma}\label{le-1.12}
Let $G$ be a digraph and $f$ a discrete Morse function on $G$.  Then any directed loop in $G$ is critical.
\end{lemma}

\begin{proof}
Let $v_0 v_1\ldots v_nv_0$  be  an arbitrary directed loop in $G$.  Suppose to the contrary, the directed loop is non-critical.  Then  by Lemma~\ref{le-1.11},   (i)'' does not hold for the directed loop,  hence (ii)'' must hold for the directed loop. That is,  there exists
a vertex $u$ of $G$ such that   $f(u)=0$ and for some $0\leq i\leq n$, $v_0 v_1\ldots v_i u v_{i+1} \ldots v_nv_0$ is a directed loop in $G$  (here we use the notation $v_{n+1}=v_0$).  This contradicts  Lemma~\ref{le-1.11}.  Therefore,  the directed loop $v_0 v_1\ldots v_nv_0$ is critical.
\end{proof}

\begin{lemma}\label{le-1.13}
Let $G$ be a digraph and $f$  a discrete Morse function on $G$. Then for any allowed elementary path in $G$, there exists at most one index such that the corresponding  vertex is with zero value.
\end{lemma}
\begin{proof}
Let $\alpha=v_0\cdots v_n$ be an allowed elementary path. Suppose to the contrary, $f(v_i)=f(v_j)=0$, $i<j$. There are two cases.

{\sc Case~1}. $v_i\not=v_j$.

Then we have $\beta=v_i\cdots v_j$, $\gamma=v_{i+1}\cdots v_j$ and $\gamma'=v_i\cdots v_{j-1}$ such that $\beta>\gamma$, $\beta>\gamma'$ and $f(\beta)=f(\gamma)=f(\gamma')$. This contradicts that $f$ is a discrete Morse function on $G$.

{\sc Case~2}. $v_i=v_j$.

{\sc Subcase~2.1}. $j=i+1$.

Then we have $f(v_iv_j)=f(v_i)=f(v_j)=0$. This contradicts that $f$ is a discrete Morse function on $G$.

{\sc Subcase~2.2}. $j\geq {i+2}$.

Then $v_iv_{i+1}\cdots v_j$ is a directed loop with $f(v_i)=0$. By Lemma~\ref{le-1.11}, this  is impossible.

Therefore, the lemma follows.
\end{proof}

\begin{lemma}\label{le-1.14}
Let $f$ be a discrete Morse function on digraph $G$. Then for  any allowed elementary path in $G$,   (i)'' and (ii)'' cannot both be true.
\end{lemma}
\begin{proof}
Let $\gamma=v_0v_1\cdots v_{n}$ be an allowed elementary path in $G$.  Suppose to the contrary, by Definition~\ref{def-1.1}, there must exist an allowed elementary $(n-1)$-path $\beta$ and an allowed elementary $(n+1)$-path $\alpha$  such that $\beta<\gamma<\alpha$,  $f(\beta)=f(\gamma)$  and $f(\alpha)=f(\gamma)$.  By Proposition~\ref{pr-1.1}, we consider the following cases.

{\sc Case~1}. There exists an allowed elementary $n$-path $\tilde{\gamma}\not=\gamma$ such that $\beta<\tilde{\gamma}<\alpha$.

Then similar to the proof of \cite[Lemma~2.5]{forman1}, by Definition~\ref{def-1.1}, we have
\begin{eqnarray*}
f(\beta)<f(\tilde{\gamma}), f(\tilde{\gamma})<f(\alpha).
\end{eqnarray*}
Thus $f(\gamma)=f(\beta)<f(\tilde{\gamma})<f(\alpha)=f(\gamma)$ which is a contradiction.

{\sc Case~2}.  There does not exist any allowed elementary $n$-path $\tilde{\gamma}\not=\gamma$ such that $\beta<\tilde{\gamma}<\alpha$.

Then $\beta$ must be obtained by removing two subsequent vertices $v_i\to v_{i+1}$ in $\alpha$ where $0\leq i \leq n$. Hence,  $f(v_i)=f(v_{i+1})$ in $\alpha$. This contradicts with Lemma~\ref{le-1.13}.

Therefore,  (i)'' and (ii)'' cannot both be true.
\end{proof}

We call an allowed elementary path $v_0v_1\ldots v_n$ {\it simplicial} if all the vertices $v_0,v_1,\ldots, v_n$ are distinct.  For each $n\geq 0$, let $S_n(G)$ be the collection of all the formal linear combinations of  simplicial allowed elementary $n$-paths in $G$.  Then $S_n(G)$ is a sub-$R$-module of $P_n(G)$.
The {\it  concatenation} of a $p$-path $\alpha=v_0v_1\cdots v_p$ and  a $q$-path $\beta=w_0w_1\cdots w_q$ is
\begin{eqnarray*}
\alpha*\beta=\left\{
\begin{array}{cc}
v_0v_1\cdots v_p w_1\cdots w_q, & \text{ if  } v_p=w_0, \\
0, & \text{ if } v_p\not=w_0.
\end{array}
\right.
\end{eqnarray*}

\begin{lemma}\label{le-4.1}
\begin{enumerate}[(a).]
\item
Any allowed elementary path $\alpha$ in $G$ is concatenations of simplicial allowed elementary paths $\beta_i$ and directed loops $\gamma_i$:
\begin{eqnarray}\label{eq-2020-7-9}
\alpha=\beta_1*\gamma_1*\beta_2*\gamma_2*\ldots * \beta_{k-1}*\gamma_{k-1}* \beta_k.
\end{eqnarray}
Here we allow each $\beta_i$ to be a single vertex in which case the corresponding concatenations would be trivial. Moreover, (\ref{eq-2020-7-9}) is unique under a certain algorithm.
\item
Let $\alpha$ be non-critical. Then in (\ref{eq-2020-7-9}),  there exists one $\beta_i$ which is non-critical. 
 \end{enumerate}
\end{lemma}

\begin{proof}
Firstly,  we write $\alpha$ as a sequence of vertices $v_0 v_1 \ldots v_n$.  If for each $v_i$ and  each  $0\leq j\leq i-1$, $v_i$ is different from   $v_j$,  then $\alpha$ is simplicial.  Thus we can let $\alpha=\beta_1$.  Otherwise,  we let $v_{i_1}$ to be the first vertex such that there exists   $0\leq j_1\leq i_1-1$  satisfying $v_{i_1}=v_{j_1}$.     Let $\beta_1=v_0v_1\ldots v_{j_1}$ and $\gamma_1=v_{j_1}v_{j_1+1}\ldots v_{i_1}$.  Then $v_0v_1\ldots v_{i_1}= \beta_1 *\gamma_1$.  Apply the same argument to $v_{i_1} v_{i_1+1} \ldots v_n$.  Since $n$ is finite, by induction,   $\alpha$ can be uniquely written as concatenations $\beta_1*\gamma_1*\beta_2*\gamma_2*\ldots * \beta_{k-1}*\gamma_{k-1}* \beta_k$ for some $k$ under this algorithm. Hence (a) follows.

Secondly, if $\alpha$ be non-critical, we consider two cases.

{\sc Case~1}. (i)'' holds  for $\alpha$.

Then by Lemma~\ref{le-1.14}, (ii)'' does not hold for $\alpha$. By Lemma~\ref{le-1.12} and Lemma~\ref{le-1.13}, there exists unique  vertex with zero value which belongs to some  $\beta_i$  in (\ref{eq-2020-7-9}) and does not belong to any directed loop. Hence $\beta_i$ is non-critical.

{\sc Case~2}. (i)'' does not hold  for $\alpha$.

Then by Lemma~\ref{le-1.14}, (ii)'' holds  for $\alpha$.  By Lemma~\ref{le-1.12} and Lemma~\ref{le-1.13},  there exists  unique vertex  $u\in V(G)$ with zero value  such that $u$ is addable in $\alpha$, but not addable in any directed loop. Thus $u$ is addable in some $\beta_i$, which makes $\beta_i$ non-critical.

Combining above cases, (b) follows.
\end{proof}

By Lemma~\ref{le-4.1},  it is  direct to see that  any digraph must satisfy one of the following conditions:
\begin{enumerate}[(A).]
\item
For each allowed elementary path $\alpha$, all $\beta_i$ in (\ref{eq-2020-7-9}) are simplicial  allowed elementary paths  satisfying that each vertex of $\beta_i$ belongs to one of the directed loops in (\ref{eq-2020-7-9});
\item
For an allowed elementary path $\alpha'$, there exists a $\beta'_i$ in (\ref{eq-2020-7-9}) satisfying that at least one vertex of $\beta_i$ does not belong to any directed loop in (\ref{eq-2020-7-9}).
\end{enumerate}

Then we claim that
\begin{proposition}\label{prop-4.11}
For  digraphs satisfying the condition (A),  all discrete Morse functions are trivial; For digraphs satisfying the condition (B), there are not only trivial functions, but also nontrivial functions.
\end{proposition}
\begin{proof}
Let $G$ be a digraph.  If $G$ satisfies the condition (A), then by Lemma~\ref{le-1.11}, each vertex of $G$ can only  be  assigned a positive real number. Hence, all discrete Morse functions on $G$ are trivial here.

If $G$ satisfies the condition (B),  we can  define trivial functions on $G$ firstly. In addition,  we can assign $0$ to one vertex which does not belong to any directed loop in (\ref{eq-2020-7-9}), and assign positive real values to other vertices on $G$.  By Definition~\ref{def-1.1}, we know that the function defined in this way is a discrete Morse function on $G$. That is, for this case, there are not only trivial functions on $G$, but also nontrivial functions.
\end{proof}


\begin{example}\label{ex-4.1}
Let $\alpha=v_0v_1v_2v_3v_4v_3v_5v_4v_2$ be an allowed elementary path of $G$. Then it can be written as $\alpha=\beta_1*\gamma_1*\beta_2$ where $\beta_1=v_0v_1v_2v_3$, $\gamma_1=v_3v_4v_3$ and $\beta_2=v_3v_5v_4v_2$, or $\alpha=\beta'_1*\gamma'_1*\beta'_2$ where $\beta'_1=v_0v_1v_2v_3v_4$, $\gamma'_1=v_4v_3v_5v_4$ and $\beta'_2=v_4v_2$. But under the algorithm of Lemma~\ref{le-4.1}(a), it can  be written as $\beta_1*\gamma_1*\beta_2$ uniquely.
\end{example}

\smallskip

\subsection{Additional properties of the discrete Morse functions}

 Let $G'$ be a sub-digraph of $G$.
\begin{itemize}
\item
Suppose there is   a  discrete Morse function  $f$  on $G$.  Then  by Definition~\ref{def-1.1},   by  defining $f'(v)=f(v)$ for any  $v\in V(G')$,   $f$  gives a discrete Morse function $f'$  on $G'$.
 Hence
 \begin{quote}
 {\it The restriction of a discrete Morse function to a sub-digraph is still a discrete Morse function. }
 \end{quote}
\item
Suppose there is a discrete Morse function $f'$ on $G'$.   The next example shows that $f'$ may {\bf not} be extendable to be a discrete Morse function on $G$.
\end{itemize}

\begin{example}
Let $V=\{v_0,v_1,v_2, v_3\}$.  Let $G'$ be a digraph with the set of vertices $V$  and the set of directed edges $\{v_0\to v_1, v_0\to v_3,  v_1\to v_2\}$.  Let $f'$ be a function on $V$  given by $f'(v_2)=f'(v_3)=0$ and $f'(v_0)=1, f'(v_1)=2$.    Then $f'$ is a discrete Morse function on $G'$.  Let $G=\{V(G),E(G)\cup \{v_0\to v_2\}\}$.
Then $G'$ is a sub-digraph of $G$.
However, $f'$ is not a discrete Morse function on $G$.  Since $V(G)=V(G')$,   there does not exist any discrete Morse function $f$ on $G$ such that the restriction of $f$ to $G'$ equals $f'$.
\end{example}

For each $n\geq 0$,  let $\text{Crit}_n(G)$ be the free $R$-module consisting of all the formal linear combinations of critical allowed elementary $n$-paths on $G$.  Then $\text{Crit}_n(G)$ is a   sub-$R$-module of $P_n(G)$.

\begin{proposition}\label{prop-1.15}
Let $G$, $G'$ both be digraphs such that $G'\subseteq G$. Let $f$ be a discrete Morse function on  $G$ and $f'=f\mid_{G'}$. Then  for each $n \geq 0$, $\text{Crit}_n(G)\cap P(G')\subseteq \text{Crit}_n(G')$.
\end{proposition}
\begin{proof}

Let $\alpha=v_0\cdots v_n$ be a  critical allowed elementary path in $G$. Assume $\alpha$ is also an allowed elementary path  in $G'$. Then for any $0\leq i\leq n$, if $d_i(\alpha)$ is allowed in $G'$, we have that $f'(d_i(\alpha))=f(d_i(\alpha))<f(\alpha)=f'(\alpha)$ by Definition~\ref{def-1.2}(i)''. Moreover, for any $u\in V(G')$, if $v_0\cdots v_j u v_{j+1}\cdots v_n$ is an allowed elementary path in $G'$, by Definition~\ref{def-1.2}(ii)'', we have that
\begin{eqnarray*}
f'(v_0\cdots v_j u v_{j+1}\cdots v_n)=f(v_0\cdots v_j u v_{j+1}\cdots v_n)>f(\alpha)=f'(\alpha).
\end{eqnarray*}
This implies the assertion.
\end{proof}

The next example  shows that the inverse of Proposition~\ref{prop-1.15} may  not  be true.
\begin{example}\label{ex-2}
Let $V=\{v_0,v_1,v_2\}$.  Let $G$ be a digraph with the set of vertices $V$ and
the set of directed edges $\{v_0\to v_1, v_1\to v_2, v_0\to v_2\}$. Let $f$ be a function on $V$ given by
$f(v_1)=0$ and $f(v_0)=1$, $f(v_2)=2$. Then $f$ is a discrete Morse function on $G$ and $\alpha={v_0v_1v_2}$ is not critical in $G$.
Let $G'$ be the digraph with the set of vertices $V$ and the set of directed edges $\{v_0\to v_1, v_1\to v_2\}$.  Then $G'$ is a sub-digraph of $G$ and $f'=f\mid_{G'}$ is a discrete Morse function on $G'$. However, $\alpha={v_0v_1v_2}$ is critical in $G'$.
\end{example}

\smallskip

\section{Discrete gradient vector fields and Morse complex}\label{se-6}
In this section, we define the discrete gradient vector fields on digraphs and prove that it is an acyclic matching.  Based on this, we give the proof of Theorem~\ref{th-0.1}.
\subsection{Discrete gradient vector fields on digraphs}

 Let $G$ be a digraph.  Let $f: V(G)\longrightarrow [0,+\infty)$ be a discrete Morse function on $G$.  For any $n\geq 0$ and any allowed elementary paths $\alpha^{(n)}<\beta^{(n+1)}$ on $G$,  if   $f(\alpha)=f(\beta)$,  then we assign a pair $(\alpha,\beta)$.   By collecting all such pairs,  we obtain a {\it partial  matching} $\mathcal{M}(G,f)$.  We call $\mathcal{M}(G,f)$ the {\it (combinatorial) discrete gradient vector field} of $f$. The properties of a discrete Morse function imply that each allowed elementary path of $G$ is in at most one pair of $\mathcal{M}(G,f)$.

For each $n\geq 0$,   by the (combinatorial) discrete gradient vector field $\mathcal{M}(G,f)$, we can construct the (algebraic) discrete gradient vector field $\text{grad}f$ which is an $R$-linear map from $P_n(G)$  to $P_{n+1}(G)$.   For an allowed elementary $n$-path  $\alpha^{(n)}$ on $G$,  if there exists an allowed elementary $(n+1)$-path $\beta^{(n+1)}$  such that  $(\alpha,\beta)\in \mathcal{M}(G,f)$,   then we set
\begin{eqnarray*}
(\text{grad}f)(\alpha)=-\langle\partial\beta,\alpha\rangle\beta.
\end{eqnarray*}
    If there does not exist such $\beta$, then we set
  \begin{eqnarray*}
  (\text{grad} f)(\alpha)=0.
  \end{eqnarray*}
  We extend $\text{grad} f$ linearly over $R$ and obtain   an $R$-linear map  $\text{grad} f: P_n(G)\longrightarrow P_{n+1}(G)$.
   We  call $\text{grad} f$ the {\it (algebraic) discrete gradient vector field} of $f$.
By Lemma~\ref{le-1.14} and  a similar argument with \cite[Theorem~6.3~(1)]{forman1},  it follows
\begin{eqnarray*}
\text{grad} f\circ \text{grad} f=0.
\end{eqnarray*}

\smallskip

\subsection{Morse complex and path homology}
Let $R$ be an arbitrary commutative ring with a unit. Let $G$ be a digraph. Consider the chain complex
\begin{eqnarray*}
\cdots\stackrel{\partial_{n+2}}{\longrightarrow}\Omega_{n+1}(G)\stackrel{\partial_{n+1}}{\longrightarrow}\Omega_{n}(G)\stackrel{\partial_{n}}{\longrightarrow}\cdots
\end{eqnarray*}
For each $n\geq0$, choose a basis $B_n(G)$ for $\Omega_n(G)$. For $b\in B_n(G)$ and $a\in B_{n-1}(G)$, define $\langle\partial_{n}b,a\rangle$ to be the coefficients of the term $a$ in the linear combination $\partial_{n}b$. Then
\begin{eqnarray*}
\partial_{n}b=\sum_{a\in B_{n-1}(G)}\langle\partial_{n}b,a\rangle a.
\end{eqnarray*}
\begin{definition}(cf. \cite[Definition~1.1]{chain})
A {\it partial matching} $\mathcal{M}\subseteq {\bigcup_{n\geq 1}B_{n-1}(G)\times B_n(G)}$ is a collection of pairs $(a,b)$ such that $a\in B_{n-1}(G)$, $b\in B_n(G)$ for some $n$, and $\langle\partial_{n}b,a\rangle$ is invertible in $R$. Write $m(b>a)=\langle\partial_{n}b,a\rangle$ for $(a,b)\in \mathcal{M}$.
\end{definition}

For each $n\geq 0$, consider the subsets of $B_n(G)$
\begin{eqnarray*}
\mathcal{U}(B_n(G))&=&\{b\in B_n(G)\mid {\text{there exists}~ a\in B_{n-1}(G)~\text{such that}~(a,b)\in \mathcal{M}}\},\\
\mathcal{D}(B_n(G))&=&\{a\in B_n(G)\mid {\text{there exists}~ b\in B_{n+1}(G)~\text{such that}~(a,b)\in \mathcal{M}}\},\\
\mathcal{C}(B_n(G))&=& B_n(G)\setminus \big(\mathcal{U}(B_n(G))\bigcup \mathcal{D}(B_n(G))\big).
\end{eqnarray*}
Given $b\in B_n(G)$ and $a\in B_{n-1}(G)$, an {\it alternating path} is a sequence
\begin{eqnarray*}
b>a_1<b_1>a_2<b_2>a_3\cdots>a_k<b_k>a
\end{eqnarray*}
such that for each $i=1,2,\cdots,k$, $(a_i,b_i)\in\mathcal{M}$. For an alternating path $p$, we write $p^{\bullet}=b$, $p_{\bullet}=a$ and define
\begin{eqnarray*}
m(p)=(-1)^k\frac{m(b>a_1)m(b_1>a_2)\cdots m(b_k>a)}{m(b_1>a_1)m(b_2>a_2)\cdots m(b_k>a_k)}.
\end{eqnarray*}

\begin{definition} (cf. \cite[Definition~1.2]{chain})
\label{def-5.1}
A partial matching $\mathcal{M}$ is called acyclic, if there does not exist any cycle
\begin{eqnarray*}
a_1<b_1>a_2<b_2>a_3< \cdots>a_k<b_k>a_{1}
\end{eqnarray*}
with $k\geq 2$ and all $b_i\in \bigcup(B_{n}(G))$ (for any $n\geq 1$) are distinct.
\end{definition}

Let $\mathcal{M}$ be an acyclic partial matching.
\begin{definition}(cf. \cite[Definition~1.4]{chain})\label{def-6.1}
The  Morse complex is defined as
\begin{eqnarray*}
\cdots\stackrel{\partial_{n+2}^{\mathcal{M}}}{\longrightarrow}C_{n+1}^{\mathcal{M}}(B_*(G))
\stackrel{\partial_{n+1}^{\mathcal{M}}}{\longrightarrow}C_{n}^{\mathcal{M}}(B_*(G))\stackrel{\partial_{n}^{\mathcal{M}}}{\longrightarrow}C_{n-1}^{\mathcal{M}}(B_*(G))
\stackrel{\partial_{n-1}^{\mathcal{M}}}{\longrightarrow}\cdots,
\end{eqnarray*}
where the $R$-module $C_n^{\mathcal{M}}(B_*(G))$  is freely generated by the elements of $\mathcal{C}(B_n(G))$,  and the boundary map is given by $\partial_n^{\mathcal{M}}(b)=\sum_{p}m(p)p_{\bullet}$ for all alternating paths $p$ with $p^{\bullet}=b$.
\end{definition}

An {\it  atom chain complex} is a chain complex $\cdots\longrightarrow 0 \longrightarrow R \stackrel{\text{id}}{\longrightarrow} R \longrightarrow 0\longrightarrow \cdots$ where the only nontrivial modules are in the dimensions $d$ and $d-1$, and the boundary map is the identity map.  Such an atom chain complex is denoted by $\text{Atom}(d)$ (cf. \cite[P.870]{chain}).

\begin{lemma}(cf. \cite[Theorem~2.1]{chain},  \cite[Theorem~2.2]{minimal})
\label{le-8.1}
Assume that we have a free chain complex with a basis $(\Omega_*, B_*)$,  and an acyclic matching $\mathcal{M}$. Then
\begin{enumerate}[(a).]
\item
$\Omega_*$ decomposes as a direct sum of chain complexes $C_*^{\mathcal{M}}(B_*(G))\bigoplus T_*$, where $T_*\simeq {\bigoplus_{(a,b)\in \mathcal{M}}\text{Atom}(\text{dim}b)}$;
\item
$H_*(\Omega_*)=H_*(C_*^{\mathcal{M}}(B_*(G)))$.
\end{enumerate}
\qed
\end{lemma}


To prove Theorem~\ref{th-0.1}, we first prove the next lemma to show that $\mathcal{M}(G,f)$  is acyclic.
\begin{lemma}\label{le-4.2}
Let $G$ be a digraph and $f$ a discrete Morse function on $G$. Then $\mathcal{M}(G,f)$ is an acyclic matching.
\end{lemma}
\begin{proof}
By considering the value of $f$ at each vertex of $G$, we separate the proof into two cases.

{\sc Case~1}. $f(v)>0$ for any vertex $v\in V(G)$.

 Then we have $\mathcal{M}(G,f)=\emptyset$.

 {\sc Case~2}. There exists a vertex $v\in V(G)$ such that  $f(v)=0$.

 Then $\mathcal{M}(G,f)$ is a nonempty finite set. That is, for each $n\geq 0$, there exist finite pairs $\{\alpha^{(n)},\beta^{(n+1)}\}$ such that $\alpha^{(n)}<\beta^{(n+1)}$ and $f(\alpha^{(n)})=f(\beta^{(n+1)})$.


Suppose to contrary, by Definition~\ref{def-5.1}, there exists a cycle
\begin{eqnarray*}
\alpha_1^{(n)}<\beta_1^{(n+1)}>\alpha_2^{(n)}<\beta_2^{(n+1)}>\alpha_3^{(n)}< \cdots>\alpha_k^{(n)}<\beta_k^{(n+1)}>\alpha_{1}^{(n)}
\end{eqnarray*}
with $k\geq 2$,  all $\beta_i^{(n+1)}$  are distinct and $(\alpha_i^{(n)}<\beta_i^{(n+1)})\in \mathcal{M}(G,f)$ ($1\leq i\leq k$) for an integer $n\geq 0$.
Since  $(\alpha_i^{(n)},\beta_i^{(n+1)})\in \mathcal{M}(G,f)$, $\beta_i^{(n+1)}$ is non-critical  for each $1\leq i\leq k$. By Lemma~\ref{le-1.14}, (ii)'' does not hold for $\beta_i^{(n+1)}$.  That is, there exists one vertex $u_i$ of $\beta_i^{(n+1)}$ such that $f(u_i)=0$  for each $1\leq i\leq k$.  We consider the following two subcases.

{\sc Subcase~1}. $\alpha_2^{(n)}=\alpha_1^{(n)}$.

Then by Definition~\ref{def-1.1}(ii), we have that $\beta_1^{(n+1)}=\beta_2^{(n+1)}$. This contradicts with that all $\beta_i^{(n+1)}$ are distinct.

{\sc Subcase~2}. $\alpha_2^{(n)}\not=\alpha_1^{(n)}$.

Then $V(\alpha_2^{(n)})=V(\beta_1^{(n+1)})\setminus \{\text{a nonzero vertex of}~\beta_1^{(n+1)}\}$ and $u_1\in V(\alpha_2^{(n)})$.
By Lemma~\ref{le-1.13}, since $\alpha_2^{(n)}<\beta_2^{(n+1)}$, it follows that $f(\alpha_2^{(n)})<f(\beta_2^{(n+1)})$.  This contradicts with $(\alpha_2^{(n)},\beta_2^{(n+1)})\in \mathcal{M}(G,f)$.

Summarising Case 1 and Case 2, we have that $\mathcal{M}(G,f)$ is  acyclic.
\end{proof}

Now we prove Theorem~\ref{th-0.1}.
\begin{proof}[Proof of Theorem~\ref{th-0.1}]
Firstly, by Proposition~\ref{prop-4.11}, we have that on any digraph there always exists a discrete Morse function.
 Secondly, since $G$ is transitive, by Proposition~\ref{prop-2.1}, we  have $\Omega_n(G)=P_n(G)$ for each $n\geq 0$. Hence, all allowed elementary paths form a basis of the chain complex $\{\Omega_*(G),\partial_*\}$.  Therefore, by Lemma~\ref{le-4.2}, by taking the allowed elementary paths as a basis of $\Omega_*(G)$, there always exists an acyclic matching. Hence by Lemma~\ref{le-8.1}(a), (\ref{eq-5.1})  follows.

Moreover, by Lemma~\ref{le-8.1}(b), the homology groups of $\Omega_{\ast}(G)$ and $C_{*}^{\mathcal{M}}(B_*(G))$ are isomorphic. Thus (\ref{eq-5.2}) follows.
\end{proof}

In addition, we can get the following theorem.
\begin{theorem}\label{th-4.4}
Let $G$ be a digraph containing neither triangle nor square. Then  both (\ref{eq-5.1}) and (\ref{eq-5.2}) hold.
\end{theorem}
\begin{proof}
By \cite[Theorem~4.3]{9}, $\text{dim}\Omega_n(G)=0$ for all $n\geq 2$. Meanwhile, $\Omega_1(G)=P_1(G)$ and $\Omega_0(G)=P_0(G)$. Hence, by Lemma~\ref{le-4.2} and \cite[Theorem~2.1]{chain}, the assertion follows.
\end{proof}

\begin{remark}
Let $G$ be a digraph (not transitive) containing triangles or squares. For each $n\geq 2$, the elements of the basis of $\Omega_n(G)$ are not only allowed elementary paths, but also linear combinations of allowed elementary paths.  The existence of acyclic matching needs further explorations.
\end{remark}

\smallskip

\section{$\mathcal{M}$-collapses by discrete gradient vector fields}\label{se-7}
Let $G$ be a digraph and $f$ a discrete Morse function on $G$. For any vertex $v\in V(G)$, we define the number of edges starting from $v$ as the {\it out-degree} of $v$, and the number of edges arriving at $v$ as the {\it in-degree} of $v$. 
Assume that  the out-degree and in-degree of any zero-point of $f$  on $G$ are both 1. In this section, we define $\mathcal{M}$-collapses and give the proof of Theorem~\ref{th-0.2}.

Let $v\in G$ be a zero-point of $f$ on $G$. By Lemma~\ref{le-1.12}, $v$ is not the vertex of any directed loop. Moreover, since the out-degree and in-degree of any zero-point of $f$ on $G$ are both 1,  there exists a unique ordered triple of vertices
$(u,v,w)$  such that $u\to v\to w$ and $u, v, w$ are distinct. 
If $u\to w$ is a directed edge in $G$, then we substitute $u\to v\to w$ with $u\to w$. Consider the digraph $G'$ whose set of vertices is $V(G)\setminus\{v\}$ and whose set of directed edges is $E(G)\setminus\{u\to v, v\to w\}$. By Definition~\ref{def-1.1}, the function $f'$ on $V(G')$ defined by $f'(x)=f(x)$ for any $x\in V(G')$ gives a discrete Morse function on $G'$. We call the pair $(G', f')$ a {\it one-step $\mathcal{M}$-collapse} of the pair $(G,f)$.

\begin{lemma}\label{le-7.11}
For any pair $(\alpha,\beta)\in \mathcal{M}(G,f)$, it is in one of the following forms:
\begin{eqnarray*}
(\alpha,\beta)=\left\{
\begin{array}{c}
\alpha=\cdots \to u\to w\to\cdots, ~\beta=\cdots \to u\to v \to w\to\cdots\\
\alpha=\cdots\to u, ~\beta=\cdots\to u\to v\\   
\alpha= w\to\cdots, ~\beta=v \to w\to \cdots   
\end{array}
\right.
\end{eqnarray*}
where $f(v)=0$ and $u\to v \to w$.
\end{lemma}
\begin{proof}
Since $(\alpha,\beta)\in \mathcal{M}(G,f)$, $\beta^{(n+1)}>\alpha^{(n)}$ and $f(\alpha)=f(\beta)$. By Lemma~\ref{le-1.13}, there must exist a unique vertex $v$  of $\beta$ such that $f(v)=0$. Since the out-degree and in-degree of any zero-point of $f$ on $G$ are both 1,
there exists a unique ordered triple of vertices $(u,v,w)$  such that $u\to v\to w$ and $u, v, w$ are distinct.
There are two cases to consider.

{\sc Case~1}.  $u\to w$ is a directed edge in $G$.

{\sc Subcase~1.1}. $u\to v \to w$ is  a subgraph of $\beta$.

Then $\alpha$ is the allowed elementary path  obtained by substituting $u\to v \to w$ with $u\to w$.

{\sc Subcase~1.2}.  $u\to v \to w$ is not a subgraph of $\beta$.

Then we must have $\beta=\cdots \to u\to v$ or $\beta= v\to w\to \cdots$. Correspondingly,  $\alpha=\cdots \to u$ or $\alpha=w\to \cdots$.




{\sc Case~2}.  $u\to w$ is not a directed edge in $G$.

{\sc Subcase~2.1}.  $u\to v \to w$ is  a subgraph of $\beta$.

Then by Lemma~\ref{le-1.13} and Definition~\ref{def-1.2}, $\alpha$ is critical. This contradicts that $(\alpha,\beta)\in \mathcal{M}(G,f)$.

{\sc Subcase~2.2}.  $u\to v \to w$ is  not a subgraph of $\beta$.

Then we must have $\beta=\cdots \to u\to v$ or $\beta= v\to w\to \cdots$. It follows that $\alpha=\cdots \to u$ and $\alpha=w\to \cdots$ respectively.





Therefore, the lemma follows.
\end{proof}

\begin{lemma}\label{le-6.12}
$\mathcal{M}(G',f')\subseteq  \mathcal{M}(G,f)$.
\end{lemma}

\begin{proof}
For any pair $(\alpha',\beta')\in \mathcal{M}(G',f')$, $\alpha', \beta'$ are both allowed elementary paths in $G'$ such that $\alpha'<\beta'$ and $f'(\alpha')=f'(\beta')$. Since $G'\subseteq G$ and $f'=f\mid_{G'}$, it follows that $\alpha',\beta'$ are also allowed elementary paths in $G$ and $f(\alpha')=f(\beta')$. Hence $(\alpha', \beta')\in \mathcal{M}(G,f)$.
The lemma follows.
\end{proof}

Since there are finite zero-points of $f$ on $G$, it follows that there are finite triples $\{(u_k, v_k, w_k)\}_{0\leq k\leq N}$ such that $f(v_k)=0$ and $u_k\to w_k$ is an allowed elementary path of $G$.  Denote the subgraph obtained by $k$-step $\mathcal{M}$-collapse as $G_k$ whose set of vertices is $V(G)\setminus\{v_1,\cdots,v_k\}$ and whose set of directed edges is $E(G)\setminus\{(u_1\to v_1, v_1\to w_1),\cdots,(u_k\to v_k, v_k\to w_k)\}$. Similarly, the restriction of $f$ on $G_k$ (denoted as $f_k$) is a discrete Morse function on $G_k$.  Obviously,  $G_N\subseteq \cdots \subseteq\cdots  G_{1}\subseteq G$. By induction, we can get a subgraph of $G$ in which there is no triple $(u,v, w)$ such that $f(v)=0$, $u\to v\to w$ and $u\to w$. We denote it as $(\tilde{G},\tilde{f})$.

By  Lemma~\ref{le-7.11} and Lemma~\ref{le-6.12}, we have that
\begin{proposition}\label{pro-6.1}
Any pair $(\alpha,\beta)\in\mathcal{M}(\tilde{G},\tilde{f})$ is in the  form of
\begin{equation*}
  \left\{
    \begin{array}{ll}
      \alpha=\cdots \to u  \\
      \beta=\cdots \to u\to v
    \end{array}
  \right.\quad  \textmd{or} \quad \left\{
    \begin{array}{ll}
      \alpha=w\to\cdots \\
      \beta=v\to w\to\cdots 
    \end{array}
  \right.
\end{equation*}
where $f(v)=0$, $u\to v\to w$ and $u\to w$ is not a directed edge in $G$.
\end{proposition}
\begin{proof}
By Lemma~\ref{le-6.12}, $\mathcal{M}(\tilde{G},\tilde{f})\subseteq \mathcal{M}(G,f)$. Note that any pair $(\alpha,\beta)\in\mathcal{M}(G,f)$ in Case1 of Lemma~\ref{le-7.11} is removed by $\mathcal{M}$-collapse. Hence the assertion follows.
\end{proof}







In the next, we prove that $\tilde {G}$ and $G$ have the same path homology groups.
\begin{proof}[Proof of Theorem~\ref{th-0.2}]
Define a digraph map $r: G\rightarrow \tilde{G}$ such that
\begin{eqnarray}\label{eq-1}
r(v)=\left\{
\begin{array}{cc}
w, & \text{if}~f(v)=0~\text{and there exists a triple}~(u,v,w)\\
   & \text{such that}~u\to v\to w~\text{and}~u\to w;\\
v, & \text{otherwise}.
\end{array}
\right.
\end{eqnarray}
By the definition of $\mathcal{M}$-collapses, it can be verified directly that $r\mid_{\tilde{G}}=\text{id}_{\tilde{G}}$ and $r$ is a retraction of $G$ onto $\tilde{G}$.  By Proposition~\ref{prop-2.2} and Theorem~\ref{th-2.1}, it is sufficient to prove  that $i\circ r\simeq \text{id}_{G}$,  where $i: \tilde{G}\rightarrow  G$ is the natural inclusion map.

Let $I_1$ be the line digraph  such that the set of vertices is $\{0,1\}$ and the set of directed edges is exactly $\{0\to 1\}$. Define a map
\begin{eqnarray*}
F: G\boxdot I_1 \rightarrow G
\end{eqnarray*}
such that $F(v,0)=v$ and $F(v,1)=(i\circ r)(v)$. Then by (\ref{eq-1}), we have that $F\mid_{G\boxdot\{0\}}=\text{id}_{G}$, $F\mid_{G\boxdot\{1\}}=i\circ r$.
Without loss of generality,  we suppose that $F$ is a one-step $\mathcal{M}$-collapse. Then there exist a vertex $v$ such that $f(v)=0$ and a unique ordered triple  of vertices $(u,v,w)$ with $u \to v \to w$ and $u\to w$ in $G$.  It follows that
\begin{eqnarray*}
&&F((u,0)\to (v,0))=u\to v,\quad F((v,0)\to (w,0))=v\to w,\quad F((u,0)\to (w,0))=u\to w,\\
&&F((u,1)\to (v,1))=u\to w,\quad F((v,1)\to (w,1))=w, \quad F((u,1)\to (w,1))=u\to w,\\
&&F((u,0)\to (u,1))=u, \quad F((v,0)\to (v,1))=v\to w,\quad F((w,0)\to (w,1))=w.
\end{eqnarray*}
Hence $F$ is  well-defined and it is a digraph map from $G\boxdot I_1$ to $G$. By Definition~\ref{def-2.10}, we have that $i\circ r\simeq \text{id}_{G}$. By Definition~\ref{def-2.11}, $r$ is  a deformation retraction. This implies the theorem.
\end{proof}




\bigskip
 \noindent {\bf Acknowledgement}. The authors would like to thank Prof. Yong Lin and Prof. Jie Wu  for their supports,  discussions and encouragements. The authors also would like to express their deep gratitude to the reviewer(s) for their careful reading, valuable comments, and helpful suggestions.

The first author is supported by the Youth Fund of Hebei Provincial Department of Education (QN2019333),
the Natural Fund of Cangzhou Science and Technology Bureau (No.197000002) and a Project of Cangzhou Normal University (No.xnjjl1902).
The second author is supported by the Postdoctoral International Exchange Program of
China 2019 project from The Office of China Postdoctoral Council, China Postdoctoral Science Foundation.

Chong Wang  (for correspondence)

 Address: $^1$School of Mathematics, Renmin University of China, 100872 China.

 $^2$School of Mathematics and Statistics, Cangzhou Normal University, 061000 China .

 e-mail:  wangchong\_618@163.com

  \medskip

 Shiquan Ren

 Address: Yau  Mathematical Sciences Center, Tsinghua University,  100084 China.

 e-mail: srenmath@126.com


\begin{thebibliography}{99}

\bibitem{ayala1}
R. Ayala, L.M. Fern$\acute{a}$ndez and J.A. Vilches,  \emph{Discrete Morse inequalities on infinite graphs}. Electron. J. Combin. {\bf 16} (1) (2009), R38.

\bibitem{ayala2}
R. Ayala, L.M. Fern$\acute{a}$ndez and J.A. Vilches,  \emph{Morse inequalities on certain infinite 2-complexes}. Glasg. Math. J. {\bf 49} (2) (2007), 155-165.

\bibitem{ayala3}
 R. Ayala, L.M. Fern$\acute{a}$ndez, D. Fern$\acute{a}$ndez-Ternero and J.A. Vilches. \emph{Discrete Morse theory on graphs}. Topol. Appl. {\bf 156} (2009), 3091-3100.

\bibitem{ayala4}
R. Ayala, L.M. Fern$\acute{a}$ndez, A. Quintero and J.A. Vilches, \emph{A note on the pure Morse complex of a graph}. Topol. Appl. {\bf 155} (2008), 2084-2089.



\bibitem{dig}
J. Bang-Jensen and G.Z. Gutin, \emph{Digraphs:
Theory, Algorithms and Applications}.  2-nd   Edition.  Springer Monographs in Mathematics, Springer, 2009.

\bibitem{hypergraph}
S. Bressan, J. Li, S. Ren and J. Wu, \emph{The embedded homology of hypergraphs and applications}. Asian J. Math. {\bf 23} (3)  (2019), 479-500.





\bibitem{9}
A. Grigor'yan, Y. Lin, Y. Muranov and S.T. Yau,  \emph{Homologies of path complexes and digraphs}. Preprint arXiv: 1207. 2834v4  (2013).

\bibitem{3}
A. Grigor'yan, Y. Lin, Y. Muranov and S.T. Yau, \emph{Cohomology of digraphs and (undirected) graphs}.  Asian J.  Math. {\bf 19} (5) (2015),  887-932.

\bibitem{4}
A. Grigor'yan, Y. Lin, Y. Muranov and S.T. Yau,
\emph{Path complexes and their homologies},  preprint (2015), https://www.math.uni-bielefeld.de/~grigor/dnote.pdf.   to appear in Int. J. Math.

\bibitem{yau2}
A. Grigor'yan,  Y. Muranov, V. Vershinin and S.T. Yau,  \emph{path homology theory of multigraphs and quivers}. Forum Math. {\bf  30} (5) (2018), 1319-1337.





\bibitem{2}
A. Grigor'yan, Y. Muranov and S.T. Yau, \emph{Homologies of digraphs and K$\ddot{u}$nneth formulas}. Commun.  Anal.    Geom. {\bf 25} (5) (2017),  969-1018.

\bibitem{10}
A. Grigor'yan, Y. Lin, Y. Muranov and S.T. Yau,  \emph{Homotopy theory for digraphs}. Pure Appl. Math. Q.    {\bf 10} (4)  (2014),  619-674.

\bibitem{forman1}
R. Forman,  \emph{Morse theory for cell complexes}.  Adv. Math. {\bf 134} (1998),  90-145.

\bibitem{forman3}
R. Forman, \emph{Discrete Morse theory and the cohomology ring}. Trans. Amer. Math. Soc. {\bf 354} (12) (2002), 5063-5085.

\bibitem{forman2}
R. Forman, \emph{A user's guide to discrete Morse theory}. S\'{e}m. Lothar. Combin {\bf 48} (2002), 35pp.

\bibitem{witten}
R. Forman, \emph{Witten-Morse theory for cell complexes}. Topology {\bf 37} (5) (1998), 945-979.


\bibitem{chain}
Dmitry N. Kozlov, \emph{Discrete Morse theory for free chain complexes}. C. R. Acad. Sci. Paris, Ser. I {\bf 340} (2005) 867-872.

\bibitem{minimal}
 M. J$\ddot{o}$llenbeck and V. Welker, \emph{Minimal resolutions via algebraic discrete Morse theory}. Memoirs
of the American Mathematical Society {\bf 923}, 2009.

\bibitem{kannan}
H. Kannan, E. Saucan, I. Roy and  A. Samal, \emph{Persistent homology of unweighted complex networks via discrete Morse theory}. Scientific Reports {\bf 9} (2019), article 13817.

\bibitem{lewiner}
T. Lewiner, H. Lopes and G. Tavares, \emph{Applications of Forman's discrete Morse theory to topology visualization and mesh compression}. Transactions on visualization and computer graphics {\bf 10} (5) (2004), 499-508, IEEE.

\bibitem{nanda}
K. Mischaikow and V. Nanda, \emph{Morse theory for filtrations and efficient computation of persistent homology}. Discrete Comput. Geom. {\bf 50} (2013), 330-353.


\end{thebibliography}
 \end{document}